\documentclass{amsart}

\usepackage[utf8]{inputenc}

\usepackage[]{amsmath}
\usepackage{amsthm}
\usepackage{latexsym}
\usepackage{amssymb}
\usepackage{mathrsfs}
\usepackage{exscale}
\usepackage{textcomp}
\usepackage[all,ps,tips,tpic]{xy}
\usepackage{upgreek}
\usepackage{url}
\usepackage{booktabs}
\usepackage[final,pdftex,colorlinks=false,pdfborder={0 0 0}]{hyperref}
\usepackage{aliascnt}

\usepackage{tikz}
\usetikzlibrary{arrows,calc,matrix,patterns}

\newcommand{\Nat}{\mathbb{N}}
\newcommand{\Int}{\mathbb{Z}}

\newcommand{\defeq}{\mathrel{\mathop:}=}

\newcommand{\ConfSink}[1]{\Conf_{#1}^\mathrm{sink}}

\numberwithin{thmcounter}{section}
\newaliascnt{thmauto}{thmcounter}

\newaliascnt{conjauto}{thmcounter}

\newaliascnt{defauto}{thmcounter}

\newaliascnt{exauto}{thmcounter}

\newaliascnt{lemauto}{thmcounter}

\newaliascnt{propauto}{thmcounter}

\newaliascnt{corauto}{thmcounter}

\newaliascnt{remauto}{thmcounter}

\theoremstyle{plain}

\newtheorem{conj}[conjauto]{Conjecture}

\newtheorem{prop}[propauto]{Proposition}

\newtheorem{thmA}{Theorem}

\theoremstyle{definition}
\newtheorem{definition}[defauto]{Definition}

\theoremstyle{remark}
\newtheorem{rem}[remauto]{Remark}
\newtheorem*{note*}{Note}

\let\originalleft\left
\let\originalright\right
\renewcommand{\left}{\mathopen{}\mathclose\bgroup\originalleft}
\renewcommand{\right}{\aftergroup\egroup\originalright}

\def\polhk#1{\setbox0=\hbox{#1}{\ooalign{\hidewidth
    \lower1.5ex\hbox{`}\hidewidth\crcr\unhbox0}}}

\DeclareMathOperator{\Conf}{Conf}

\DeclareMathOperator{\Star}{Star}
\DeclareMathOperator{\HH}{H}
\DeclareMathOperator{\map}{map}

\newcommand{\particle}[1]{\begin{scope}[shift={#1}]\draw[fill] (0,0) circle (0.1);\end{scope}}
\newcommand{\LST}[1]{\mathbf{#1}}

\usepackage[pdftex]{hyperref}
\usepackage{microtype}
\usepackage[draft]{fixme}
\usepackage{soul}

\usepackage{tikz}
\usepackage{pgfplots}
\usepgfplotslibrary{polar}

\newcommand{\particleNr}[2]{\begin{scope}[shift={#1}]\node[fill=white,draw,circle,inner sep=1pt] at
  (0,0) {\small #2};\end{scope}}

\begin{document}
\title{The Homology of Configuration Spaces of Trees with Loops}

\author{Safia Chettih}
\address{Department of Mathematics, Reed College, Oregon, USA}
\email{safia@reed.edu}

\author{Daniel Lütgehetmann}
\address{Institut für Mathematik, Freie Universität Berlin, Germany}
\email{daniel.luetgehetmann@fu-berlin.de}


\begin{abstract}
  We show that the homology of ordered configuration spaces of finite trees with loops is torsion
  free.
  We introduce configuration spaces with sinks, which allow for taking quotients of the base
  space.
  Furthermore, we give a concrete generating set for all homology groups of configuration spaces of
  trees with loops and the first homology group of configuration spaces of general finite graphs.
  An important technique in the paper is the identification of the $E^1$-page and differentials of
  Mayer-Vietoris spectral sequences for configuration spaces.
\end{abstract}

\maketitle

\section{Introduction}
For a topological space $X$ and a finite set $S$ we define the \emph{configuration space of
$X$ with particles labeled by $S$} as
\[ \Conf_S(X) \defeq \left\{f\colon S\to X \text{ injective} \right\} \subset \map(S, X). \]
For $n\in\Nat$ we write $\mathbf{n}\defeq \{1, 2, \ldots, n\}$ and $\Conf_n(X)\defeq
\Conf_{\mathbf{n}}(X)$.
This is usually called the $n$-th ordered configuration space of $X$.
Let $G$ be a finite connected graph (i.e.\ a connected 1-dimensional CW complex with finitely many
cells). We are interested in the homology of configurations of $n$ ordered particles in $G$, that
is, $H_*(\Conf_n(G))$.

\vspace{1ex}
A main ingredient in proving results about configurations in graphs is the existence of
combinatorial models for the configuration spaces.
In \cite{Abrams00}, Abrams introduced a discretized model for the configuration space of $n$ points
in a graph which is a cubical complex, allowing the spaces to be studied using techniques from
discrete Morse theory and connecting them with right-angled Artin groups (see \cite{Farley05},
\cite{Crisp04}).
A similar discretized model for non-$k$-equal configuration spaces in a graph, where up to $k-1$
points are allowed to collide, was constructed in \cite{Chettih16}, providing inspiration for the
configuration with sinks introduced in this paper.

Not long after the introduction of Abrams' model, {\'S}wi{\polhk{a}}tkowski introduced a cubical
complex which is a deformation retract of the space of unordered configurations of $n$ points in a
graph (see \cite{Swiat01}).
In this model, instead of the points moving discrete distances along the graph, the points move from
an edge to a vertex of valence at least two or vice versa.
This gives a sharper bound for the homological dimension of these configuration spaces as the
dimension of the complex is bounded from above by the number of vertices in the graph (see \cite{Ghrist01}, \cite{Farley05} for proofs that this bound also holds for Abrams' model).
An analogous model holds for ordered configurations (see \cite{Luetgehetmann14}), by keeping track
of the order of points on an edge.
The combinatorial model for configurations with sinks has structure similar to the latter models.

In order to describe the homology of $\Conf_n(G)$ we will compare it to a modified version of
configuration spaces: we add ``sinks'' to our graphs.
Sinks are special vertices in the graph where we allow particles to collide.
For ordinary configuration spaces, if we collapse a subgraph $H$ of $G$ then this does \emph{not} induce a map
\[ \Conf_n(G)\dashrightarrow\Conf_n(G/H)\]
because some of the particles could be mapped to the same point in $G/H$.
If, however, we turn the image of $H$ under $G\to G/H$ into a sink, there is now an induced map on
configuration spaces.

Our first theorem shows that in the \emph{ordered} case, there is no torsion and a geometric
generating system for a large class of finite graphs.

\begin{definition}
  A finite connected graph $G$ is called a \emph{tree with loops} if it can be constructed as an
  iterated wedge of star graphs and copies of $S^1$.
\end{definition}

\begin{definition}
  A homology class $\sigma\in H_q(\Conf_n(G))$ is called the \emph{product of classes $\sigma_1\in
    H_{q_1}(\Conf_{T_1}(G_1))$ and $\sigma_2\in H_{q_2}(\Conf_{T_2}(G_2))$} for $q_1+q_2=q$ if it is
    the image of $\sigma_1\otimes \sigma_2$ under the map
  \[
    H_q(\Conf_n(G_1\sqcup G_2)) \to H_q(\Conf_n(G))
  \]
  induced by an embedding $G_1\sqcup G_2\hookrightarrow G$.
  Analogously, iterated products are induced by embeddings $G_1\sqcup G_2 \sqcup \ldots \sqcup G_n \hookrightarrow G$.

  For $k\ge 3$ let $\Star_k$ be the star graph with $k$ leaves, $\HH$ the tree with two vertices
  of valence three and $S^1$ a circle with one vertex of valence 2.
  We call a class $\sigma\in H_q\left( \Conf_n(G) \right)$ a \emph{product of basic classes} if
  $\sigma$ is an iterated product of classes in groups of the form $H_j(\Conf_{n_i}(G_i))$ where
  $j$ equals 0 or 1 and $G_i$ is a star graph, the $\HH$-graph, the circle $S^1$ or the interval
  $I$.
\end{definition}

\begin{thmA}
  \label{thm:trees}
  Let $G$ be a tree with loops and let $n$ be a natural number.
  Then the integral homology $H_q\left( \Conf_n(G); \Int \right)$ is torsion-free and generated by
  products of basic classes for each $q\ge0$
\end{thmA}

A 1-class in $S^1$ moves all particles around the circle, a 1-class in a star graph uses the
essential vertex to shuffle around the particles, and a 1-class in the $\HH$-graph uses one of the
vertices to reorder the particle and then undoes this reordering using the other vertex.
The proof of \autoref{thm:trees} will show that 2-classes in an $\HH$-graph are given by sums of
products of 1-classes in the two stars, and there are no higher dimensional classes in these three
types of graphs.

The proof of \autoref{thm:trees} rests on an inductive argument on the number of essential vertices
of a graph.
We construct a basis for the configuration space of a star graph with loops such that the $E^1$-page
of the Mayer-Vietoris spectral sequence induced by our gluing splits over that basis.
We can identify a part of the homology of the $E^1$-page with configuration spaces where some of the
points have been forgotten, and the rest of the homology with a configuration space where the star
graph has been collapsed to a sink (see \autoref{sec:conf-with-sinks} for the definition of sink
configuration spaces).
The gluing process does not create torsion, so torsion-freeness follows from explicit calculations
of the homology of ordered configurations in star graphs with loops.
An explicit generating set of homology classes with known relations is essential to our proof.
A basis for the homology of ordered configurations of two points in a tree was first constructed in
\cite{Chettih16}, which highlighted the role of basic classes of the $\HH$ graph in the
configuration space of wedges of graphs. See also \cite{BF09} and \cite{FH10} for descriptions of
product structure in configurations of two points on planar and non-planar graphs.
The Mayer-Vietoris principle was previously used to compute the homology of (unordered)
configuration spaces of graphs in \cite{MaSa16}.

For more general graphs, the analogous theorems do not hold:
\begin{thmA}\label{thm:non-product-general-graph}
  If $G$ is any finite graph and $n$ a natural number, then the \emph{first} homology group
  $H_1(\Conf_n(G))$ is generated by basic classes.
  However, for each $i\ge2$ there exists a finite graph $G$ and a number $n$ such that
  $H_i(\Conf_n(G))$ is not generated by products of 1-classes.
\end{thmA}
We provide explicit examples for the second statement.
Abrams and Ghrist were aware of the second part of this result in 2002 (\cite{AbramsGhrist02}), but
their example does not generalize to arbitrary dimensions.
More specifically, they showed that $\Conf_2(K_5)$ and $\Conf_2(K_{3,3})$ are homotopic to surfaces
of genus $6$ and $4$ respectively, where $K_5$ is the complete graph on five vertices and $K_{3,3}$
is the complete bipartite graph on $3+3$ vertices.

\vspace{1em}

Both theorems above can be generalized to the case where arbitrary subsets of the vertices are
turned into sinks.

In between versions of this paper, Ramos considered configurations where all the vertices of a graph
are sinks, approaching them through the lens of representation stability (\cite{Ramos17}). His
theorems concerning torsion-freeness and bounds on homological dimension are special cases of the
theorems above.

In an earlier version of this paper, we asserted torsion-freeness for arbitrary finite graphs.
However, our proof relied on a basis which we discovered does not split the Mayer-Vietoris spectral
sequence in the way we described it.
Our investigation of obstructions to constructing an appropriate basis led us to the counterexamples in \autoref{thm:non-product-general-graph}.
Such a basis may still exist, and we believe the following:
\begin{conj}\label{conj:torsion-free}
Let $G$ be a finite graph and $n$ a natural number. Then the integral homology
$H_q(\Conf_n(G);\Int)$ is torsion-free for each $q \geq 0$.
\end{conj}
To answer this question for general graphs, more work is needed on relations in the homology of
configuration spaces of graphs with many cycles.

The paper proceeds as follows: we introduce a combinatorial model for configurations with ``sinks''
in order to calculate the homology of a few specific examples in \autoref{sec:conf-with-sinks}.
After the Mayer-Vietoris spectral sequence is established in \autoref{sec:mv-spectral-sequence}, we
construct our desired basis and argue inductively by gluing on stars with loops in
\autoref{sec:trees}.
The case of the first homology in an arbitrary graph comprises
\autoref{sec:general-graph}, with counterexamples for higher homology.
Our techniques in this section are substantially different since we no longer have bases which
split the spectral sequence.

\subsection{Acknowledgements}
The second author was supported by the Berlin Mathematical School and the SFB 647 “Space – Time –
Matter” in Berlin.
The authors want to thank Elmar Vogt and Dev Sinha for helpful discussions, and the referee, whose
comments helped us make the paper more readable.

\section{Quotient and Mayer-Vietoris constructions}\label{sec:conf-with-sinks}
In order to describe the homology of $\Conf_n(G)$ we will compare it to a modified version of
configuration spaces: we add ``sinks'' to our graphs.
Sinks are special vertices in the graph where we allow particles to collide, and they enable us to
collapse subgraphs and get an induced map on configuration spaces.
This does not work for ordinary configuration spaces:
if we collapse a subgraph $H$ of $G$ then this does \emph{not} induce a map
\[ \Conf_n(G)\dashrightarrow\Conf_n(G/H)\]
because some of the particles could be mapped to the same point in $G/H$.

For a number $n\in\Nat$, a graph $G$ and a subset $W$ of $G$'s vertices we define the following
configuration space with sinks:
\begin{equation*}
  \ConfSink{n}(G,W) = \left\{(x_1,\ldots,x_n)\,|\, \text{for $i\neq j$ either $x_i\neq x_j$ or
  $x_i=x_j\in W$}\right\}\subset G^n.
\end{equation*}
Looking at the collapse map $G\to G/H$ again, there is now an induced map on configuration spaces if
we turn the image of $H$ under $G\to G/H$ into a sink:
\[ \Conf_n(G) \to \ConfSink{n}(G/H, H/H).\]

\subsection{A combinatorial model}
We can extend the techniques of \cite{Swiat01} and \cite{Luetgehetmann14} to obtain a cube complex
model of configuration spaces with sinks.
More precisely we will define a deformation retraction $r\colon
\ConfSink{n}(G,W)\to\ConfSink{n}(G,W)$ such that the image of $r$ has the structure of a finite cube
complex.
Each axis of such a cube will correspond to the combinatorial movement of one particle.
A combinatorial movement here is either given by the movement from an essential non-sink vertex onto
an edge or along a single edge from one sink to the other.
Each vertex and each such edge can only be involved in one of those combinatorial movements at a
time, so the dimension of this cube complex will be restricted by the number of essential
non-sink vertices and the edges connecting two sinks.

\begin{definition}[{Cube Complex, see \cite[Definition I.7.32]{bh10}}]
    A cube complex $K$ is the quotient of a disjoint union of cubes
    $X=\bigsqcup_{\lambda\in\Lambda}[0,1]^{k_\lambda}$ by an equivalence relation $\sim$ such that
    the quotient map $p\colon X\to X/\!\!\sim\ = K$ maps each cube injectively into $K$ and we only
    identify faces of the same dimensions by an isometric homeomorphism.
\end{definition}
\begin{rem}
  The definition above differs slightly from the original definition by Bridson and Häfliger, in
  that it allows two cubes to be identified along more than one face. This is a necessary property
  for the complex we wish to describe.
\end{rem}

\begin{prop}\label{prop:combinatorial-model-sinks}
  Let $G$ be a finite graph, $W$ a subset of the vertices and $n\in\Nat$.
  Then $\ConfSink{n}(G,W)$ deformation retracts to a finite cube complex of dimension $\min\{n,
  |V_{\ge2}| + |E_W| \}$, where $V_{\ge 2}$ is the set of non-sink vertices of $G$ of valence at
  least two and $E_W$ is the set of edges incident to two sinks.
\end{prop}
\begin{proof}
  The naive approach would be to retract particles in the interior of an edge to positions
  equidistant throughout the edge.
  However, this fails to be continuous as the number of particles in the interior changes, such as
  when a particle moves off a vertex.
  To fix this, we construct an additional parameter which controls the distance of the outermost
  particles on an edge from the vertices.

  Give $G$ the path metric such that every edge has length 1.
  For this proof, we define half edges in $G$:
  every edge consists of two distinct half edges $h^\iota_e$ and $h^\tau_e$.
  For each half edge $h$ we denote by $v(h)$ the vertex incident to $h$ and by $e(h)$ the edge
  corresponding to $h$ \emph{with the orientation determined by the half edge}.
  If $h$ is a half edge, then $\overline{h}$ is the other half of $e(h)$, and $e(h) = -
  e\left(\overline{h}\right)$.

  The general idea is now the following:
  the retraction $r$ only changes the position of particles \emph{inside} (closed) edges of the
  graph.
  We move as many particles of a given configuration $\mathbf{x}=(x_1,\ldots,x_n)$ as possible into
  the sinks, so that $r(\mathbf{x})$ has at most one particle in the interior of any edge incident
  to a sink.
  Furthermore, the particles of $r(\mathbf{x})$ on each single edge will be equidistant, except for
  the outermost particles, which may be closer to the vertices, see
  \autoref{fig:equidistant-particles-interval}.
  The main difficulty will be to define for each configuration $\mathbf{x}$ and each half edge $h$
  the parameter $t_h\in[0,1]$ determining the distance of the particles from the corresponding
  vertex.
  Decreasing $t_h$ to zero represents moving the particle on the edge that is nearest to the vertex
  $v(h)$ towards that vertex.
  To avoid multiple particles approaching the same vertex, we therefore require that for any pair of
  half edges $h\neq h'$ with $v(h)=v(h')$ only one of the two values $t_h$ and $t_{h'}$ can be
  strictly smaller than 1.

  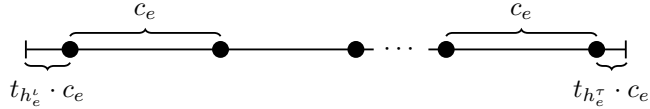
\begin{figure}[ht]
    \begin{center}
      \begin{tikzpicture}
        [circle dotted/.style={dash pattern=on .05mm off 1mm, line cap=round}, line width = .7pt]
        \node[align=center] (qdots) at (1,0) {$\cdots$};
        \path (-4,0) edge[|-] (qdots.west)
            (qdots.east) edge[-|] (4,0);
        \particle{(-3.4,0)};
        \particle{(-1.4,0)};
        \particle{(0.4,0)};
        \particle{(1.6,0)};
        \particle{(3.6,0)};
        \draw [decoration={ brace, mirror, raise=0.5em }, decorate ]
          (-4,0) -- node[below = .8em, pos=0.5] {$t_{h^\iota_e}\cdot c_e$}(-3.4,0);
        \draw [decoration={ brace, raise=0.5em }, decorate ]
          (-3.4,0) -- node[above = .8em, pos=0.5] {$c_e$}(-1.4,0);
        \draw [decoration={ brace, raise=0.5em }, decorate ]
          (1.6,0) -- node[above = .8em, pos=0.5] {$c_e$}(3.6,0);
        \draw [decoration={ brace, mirror, raise=0.5em }, decorate ]
          (3.6,0) -- node[below = .8em, pos=0.5] {$t_{h^\tau_e}\cdot c_e$}(4,0);
    \end{tikzpicture}
    \end{center}
    \caption{Equidistant particles on $e$.}
    \label{fig:equidistant-particles-interval}
  \end{figure}

  \vspace{1em}

  For fixed $(x_1, \ldots, x_n)\in\ConfSink{n}(G,W)$ we now define the image $r(\mathbf{x})$.
  The first step is to construct the parameter $t_h$.
  Let $v\in V(G)$ be a vertex and denote by $H_v$ the set of half edges $h$ with $v(h)=v$.
  If $H_v$ has only one element, then we set $t_h=1$ because we do not want to move particles
  towards a vertex of valence 1.
  Also, if $v$ is occupied by one of the particles $x_i$, then we set $t_h=1$ for all $h\in H_v$
  because we do not want to move particles towards an occupied vertex.

  Now assume that the valence of $v$ is at least two and it is not occupied by a particle.
  If for a half edge $h\in H_v$ the edge $e(h)$ contains no particles or $v(h)$ is a sink, set
  $t_h=1$.
  Otherwise, the particles on $e(h)$ cut the edge into segments, and we order these segments
  according to the orientation of $e(h)$ given by $h$.
  Let $\ell_h$ be the quotient of the length of the first segment by the length of the
  second segment, capped to the interval $[0,1]$, unless $v\left(\overline{h}\right)$ is a sink. If it is a sink, let $\ell_h$ be the length of the first segment. We treat these cases differently because particles on edges incident to sinks move from vertex to vertex instead of from edge to vertex.
  We now define
  \[
    t_h \defeq \min\left\{1, \frac{\ell_h}{\displaystyle\min_{h'\in H_v-\{h\}} \ell_{h'}}\right\}.
  \]
  Notice:
  \begin{itemize}
    \item if $\ell_h=\ell_{h'}$, then $t_h=t_{h'}=1$,
    \item if only one of the $\ell_h$ goes to zero, then also $t_h$ goes to zero, and
    \item at most one of the $t_h$ for $h\in H_v$ is strictly smaller than 1.
  \end{itemize}

  \vspace{1em}

  Given these parameters $t_h$ for all half edges $h$ we now construct the configuration $r((x_1,
  \ldots, x_n))$.
  The particles on the vertices are not moved by the retraction, so it remains to describe the
  change of position for the particles in the interior of an edge $e$.
  We will not change the order of the particles but only their position within the edge, and to make
  the description more concise we choose once and for all an isometric identification of each edge
  $e$ with $[0,1]$ such that $v(h^\iota_e)=0$.

  \textbf{If $\mathbf{e}$ is not incident to a sink vertex} the new position of the $j$-th vertex on
  $e$ will be given by $(t_{h^\iota_e} + j-1)\cdot c_e$, where $k_e\ge 1$ is the number of particles
  in the interior of $e$ and $c_e\defeq (t_{h^\iota_e} + k_e - 1 + t_{h^\tau_e})^{-1}$ will be the
  distance between the particles on that edge.
  This gives all particles on the edge the same distance and only modifies the distances from the
  vertices, see \autoref{fig:equidistant-particles-interval}.
  It remains to be shown that the positions of the particles on the edge vary continuously as $t_h$
  goes to 0. This is true when $t_{h^\iota_e}>0$, and notice that for $t_{h^\iota_e}=0$ the images
  of the particles will be the same as if we considered the first particle to be on $v(h^\iota_e)$
  and $t_{h^\iota_e}=1$:
  this would change $t_{h^\iota_e}$ from $0$ to $1$ and reduce $k_e$ by one, so that $c_e$ will be
  exactly the same.
  The analogous result also holds for $h^\tau_e$.
  This shows that the position of the particles on this closed edge after applying $r$ is continuous
  in the original configuration.

  \textbf{If $\mathbf{e}$ is incident to precisely one sink vertex} then we can assume that this
  sink vertex corresponds to $0\in[0,1]$.
  All particles on $e$ except the last one are then moved to $0$, the last particle is moved to
  $1-t_{h^\tau_e}\in[0,1]$.

  \textbf{If both vertices incident to $\mathbf{e}$ are sinks} we slide all particles away from
  $1/2\in[0,1]$ with speed given by their distance from $1/2$ until at most one particle is left in
  the interior $(0,1)$ of the interval.
  This gives a configuration having one particle on $e$ and the rest on the sinks.

  \vspace{1em}

  The map described above is continuous and a retraction, i.e.\ satisfies $r^2=r$.
  In the description we only changed the positions of particles on individual edges, so there is an
  obvious homotopy from the identity to $r$ by just adjusting the positions of the particles on each
  edge individually.

  \vspace{1em}

  The image of $r$ has the structure of a cube complex:
  the 0-cells are configurations where all particles in the interior of each interval cut the
  interval into pieces of equal length, and additionally no particle is in the interior of an edge
  with one or two sinks.
  A $k$-cube is given by choosing such a $0$-cell, $k$ distinct particles which are either outmost
  on their edge or on a sink and move them to an adjacent vertex.
  Such a choice of $k$ movements determines a $k$-cube if and only if we can realize the movements
  independently, namely if
  \begin{itemize}
    \item no two particles move along the same edge,
    \item no two particles move towards the same \emph{non-sink} vertex and
    \item no particle moves towards an occupied \emph{non-sink} vertex.
  \end{itemize}
  Each direction of the cube corresponds to the movement of one of the particles.
  By the description of the choices involved for finding $k$-cubes we immediately get the
  restriction on the dimension.
  For more details about the general construction of the cube complex (without sinks), see
  \cite{Luetgehetmann14}.
\end{proof}

It will be useful for subsequent proofs to have a notion for pushing in new particles from the
boundary of the graph.
\begin{definition}
  Let $G$ be a graph and $e$ be a leaf.
  For a finite set $S$ and an element $s\in S$, define the map
  \[
    \iota_{e,s}\colon\ConfSink{S-\{s\}}(G,W)\hookrightarrow\ConfSink{S}(G,W)
  \]
  by slightly pushing in the particles on $e$ and putting $s$ onto the univalent vertex of $e$.
\end{definition}

\begin{definition}
  Let $G$ be a graph.
  For finite sets $S'\subset S$, define the map
  \[
    \pi_{S'}\colon \ConfSink{S}(G,W)\to\ConfSink{S'}(G,W)
  \]
  by forgetting the particles $S-S'$.
  If $S'=\{s\}$ then we write instead $\pi_s\defeq \pi_{\{s\}}$.
\end{definition}

Notice that the composition $\pi_{S-\{s\}}\circ \iota_{e,s}$ is homotopic to the identity.

\begin{definition}
  Let $X=\Sigma_i \alpha_iX_i$ be a cellular chain in the combinatorial model of $\ConfSink{S}(G,
  W)$ .
  The particle $s$ is called a \emph{fixed particle of $X$} if there exists a cell $c$ of the graph
  $G$ such that $\pi_s(X_i)$ is contained in the interior of $c$ for all $X_i$.
  Here, the interior of a vertex is the vertex itself.
\end{definition}
Notice that fixed particles may still move inside their edge to preserve equidistance, but they
never leave their edge or vertex.

\subsection{The homology for small graphs}
For later use we calculate the homology of some of these configuration spaces with sinks.

\begin{prop}\label{prop:homology-sinks}
  \begin{align*}
    H_i\left( \ConfSink{n}(I, \varnothing) \right) &=
    \begin{cases}
      \Int\Sigma_n & i = 0\\
      0 & \text{else}
    \end{cases}\\
    H_i\left( \ConfSink{n}(S^1, \varnothing) \right) &=
    \begin{cases}
      \Int\left(\Sigma_n/\mathrm{shift}\right)\cong \Int^{(n-1)!} & i = 0,1\\
      0 & \text{else}
    \end{cases}\\
    H_i\left( \ConfSink{n}(I, \{0\}) \right) &=
    \begin{cases}
      \Int & i = 0\\
      0 & \text{else}
    \end{cases}\\
    H_i\left( \ConfSink{n}(I, \{0, 1\}) \right) &=
    \begin{cases}
      \Int & i = 0\\
      \Int^{(n-2)2^{n-1}+1} & i = 1\\
      0 & \text{else}
    \end{cases}\\
    H_i\left( \ConfSink{n}(S^1, \{0\}) \right) &=
    \begin{cases}
      \Int & i = 0\\
      \Int^{n} & i = 1\\
      0 & \text{else}
    \end{cases}
  \end{align*}
\end{prop}
\begin{proof}
  The first two are clear.
  The interval with one sink has contractible configuration space: we can just gradually pull all
  particles into the sink.
  For the last two cases, note that the spaces are obviously connected by pulling all particles onto
  one of the sinks.
  Furthermore, by \autoref{prop:combinatorial-model-sinks} they are homotopic to 1-dimensional cube
  complexes.
  Computing the Euler characteristic gives the described ranks:

  \vspace{1em}

  \noindent
  \ul{$\chi(\ConfSink{n}(I,\{0,1\}))$:}
  There is a zero cube for every distribution of particles onto the two sinks, which means that
  there are $2^n$ of them.
  We have a 1-cell for each choice of one moving particle and every distribution of the remaining
  ones onto the two sinks, so there are $n2^{n-1}$ many 1-cells.
  Notice that this is the 1-skeleton of the $n$-dimensional cube.
  Thus, the Euler characteristic is $(2-n)2^{n-1}$, which determines the rank of the first homology
  group.

  \noindent
  \ul{$\chi(\ConfSink{n}(S^1,\{0\}))$:}
  There is precisely one zero cell, namely the one where all particles are on the sink.
  There is one 1-cell for each choice of one particle moving along the edge, giving $n$ 1-cells and
  therefore the Euler characteristic $1-n$.
  Notice that this is a bouquet of circles.
\end{proof}

\begin{rem}\label{rem:sink-cycles-as-H-cycles}
  Cycles in $H_1(\ConfSink{n}(I, \{0, 1\}))$ can be regarded as cycles in the ordinary configuration
  space of the H-graph $\Conf_n(\HH)$, see \autoref{fig:configurations-with-sinks-in-h-graph}.
  Replace both spaces by their combinatorial models and define a continuous map as follows:
  take a 0-cell of the configuration space with sinks and replace particles sitting on a sink vertex
  with them sitting on the corresponding lower leaf of the H-graph in their canonical ascending
  order.
  Moving a particle $x$ from one sink vertex to the other is then given by moving all particles
  blocking $x$'s path to the vertex to the upper leaf, moving $x$ onto the horizontal edge, moving
  the particles on the upper leaf back to the lower leaf and repeating the same game on the other
  side in reverse.
  This determines a continuous map between combinatorial models and thus induces a map on cellular
  1-cycles.
  \begin{figure}[htpb]
    \centering
    \begin{tikzpicture}
      \draw[-] (-1,0) -- (1,0);
      \filldraw (-1,0) circle (.1cm);
      \filldraw (1,0) circle (.1cm);

      \particleNr{(-1, -.5)}{3}
      \particleNr{(-1, -.8)}{5}
      \particleNr{(-1, -1.1)}{6}

      \particleNr{(1, -.5)}{2}
      \particleNr{(1, -.8)}{4}

      \particleNr{(-.1, 0)}{1}
      \draw[->] (.15, .1) -- (.4, .1);
      \draw[->] (.15, -.1) -- (.4, -.1);

      \node at (2.3,0) {$\leftrightsquigarrow$};

      \begin{scope}[shift={(4,0)}]
        \draw[-] (-1, 1) -- (0,0);
        \draw[-] (-1, -1) -- (0,0);
        \draw[-] (0, 0) -- (2,0);
        \draw[-] (2, 0) -- (3,1);
        \draw[-] (2, 0) -- (3,-1);

        \particleNr{(-.85, -.85)}{6};
        \particleNr{(-.52, -.52)}{5};
        \particleNr{(-.2, -.2)}{3};

        \particleNr{(2.7, -.7)}{4};
        \particleNr{(2.3, -.3)}{2};

        \particleNr{(.5, 0)}{1};
        \draw[->] (.75, .1) -- (1, .1);
        \draw[->] (.75, -.1) -- (1, -.1);
      \end{scope}

    \end{tikzpicture}
    \caption{Comparing $\ConfSink{n}(I, \{0,1\})$ and $\Conf_n(\HH)$}
    \label{fig:configurations-with-sinks-in-h-graph}
  \end{figure}
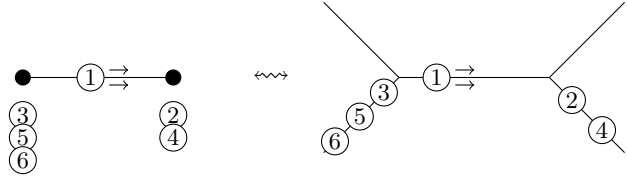

  This map is injective in homology:
  composing the map with the map collapsing the two pairs of leaves to sinks gives a map that is
  homotopic to the identity, showing that the homology of $\ConfSink{n}(I,\{0,1\})$ is a direct
  summand of the homology of $\Conf_n(\HH)$.
\end{rem}

\subsection{A Mayer-Vietoris spectral sequence for configuration spaces}
\label{sec:mv-spectral-sequence}
To compute the homology of the configuration space of a space $X$ we can decompose $X$ into smaller
spaces and patch together local results.
A structured way to do this is by using the Mayer-Vietoris spectral sequence associated with a
countable open cover.
\begin{definition}[Mayer-Vietoris spectral sequence]
  Let $J$ be a countable ordered index set and $\{V_j\}_{j\in J}$ an open cover of $X$, then we
  define the following countable open cover $\mathcal{U}(\{V_j\})$ of $\Conf_n(X)$:
  for each $\phi\colon \mathbf{n}\to J$ we define $U_\phi$ to be the set of all those configurations
  where each particle $i$ is in $V_{\phi(i)}$, i.e.
  \[
    U_\phi \defeq \bigcap_{i\in\mathbf{n}} \pi_i^{-1}\left( V_{\phi(i)} \right).
  \]
  These sets are open and cover the whole space, so they define a spectral sequence
  \[
    E^1_{p,q} = \bigoplus_{\{\phi_0,\ldots,\phi_p\}} H_q\left( U_{\phi_0}\cap\cdots\cap U_{\phi_p}
    \right) \Rightarrow H_*\left( \Conf_n(X) \right)
  \]
  converging to the homology of the whole space.
  For a proof of the convergence of this spectral sequence, see \cite[Proposition 2.1.9, p.
  13]{Chettih16}.
\end{definition}
Notice that
\[
  U_{\phi_0}\cap\cdots\cap U_{\phi_p} = \bigcap_{i\in\mathbf{n}} \bigcap_{0\le j\le p}
  \pi_i^{-1}\left( V_{\phi_j(i)} \right).
\]
For brevity, we will also write
\[
  U_{\phi_0\cdots\phi_p} \defeq U_{\phi_0} \cap \cdots \cap U_{\phi_p}.
\]
The boundary map $d_1$ is given by the alternating sum of the face maps induced by
\[
  U_{\phi_0}\cap\cdots\cap U_{\phi_p} \hookrightarrow U_{\phi_0}\cap\cdots\cap
  \widehat{U_{\phi_i}}\cap\cdots\cap U_{\phi_p}
\]
forgetting the $i$-th open set from the intersection.
Of course, this construction generalizes to configuration spaces with sinks.

\section{Configurations of particles in trees with loops}\label{sec:trees}
We will more generally prove \autoref{thm:trees} for all graphs as in the statement of the theorems
with any (possibly empty) subset of the vertices of valence one turned into sinks.
The proof will proceed by induction over the number of essential vertices (i.e.\ vertices of valence
at least three).
We first prove the base case:
\begin{prop}\label{prop:tree-base-case}
  Let $G$ be a finite connected graph with precisely one essential vertex and $W$ a subset of the
  vertices of valence 1.
  Then $H_1(\ConfSink{n}(G,W))$ is free and generated by basic classes.
\end{prop}

Notice that if we talk of $\HH$-classes in a graph \emph{with sinks} $(G,W)$ then we allow some of
the leaves of $\HH$ to be collapsed to a sink under the map $\HH\to G$.
In the proof, we will need the following definition:

\begin{definition}\label{def:configurations-of-tuples}
  For finite sets $T\subset S$, a finite graph $G$, a subset $K\subset G$, and sinks $W\subset V(G)$
  write $\Gamma=(G,K)$ and define
  \[
    \ConfSink{S,T}( \Gamma, W) = \{ f\colon S\to G\,|\, f(T)\subset K \}\subset
    \ConfSink{S}(G,W).
  \]
\end{definition}

As a consequence of the definition, we get
\[ \ConfSink{S, \emptyset}(\Gamma, W) = \ConfSink{S}(G, W) \]
and
\[ \ConfSink{S, S}(\Gamma, W) = \ConfSink{S}(K, W\cap K). \]

\begin{proof}[{Proof of \autoref{prop:tree-base-case}}]
  By \autoref{prop:combinatorial-model-sinks}, $\ConfSink{n}(G,W)$ is homotopy equivalent to a
  graph, so the first homology is free.
  To see that it is generated by basic classes, we inductively use a Mayer-Vietoris long exact
  sequence.

  For a sink $w\in W$ let $\Gamma_w = (G,G-\{w\})$.
  Notice that
  \[ \ConfSink{S,\emptyset}(\Gamma_w,W) = \ConfSink{S}(G,W). \]
  and
  \begin{align*}
    \ConfSink{S, S}(\Gamma_w,W) &= \ConfSink{S}(G-\{w\}, W-\{w\}) \\
    &\simeq \ConfSink{S}(G, W-\{w\}),
  \end{align*}
  where the last homotopy equivalence follows because $w$ has valence 1.
  For two sinks $w_0\neq w_1$ we therefore have
  \[ \ConfSink{S, S}(\Gamma_{w_0},W) \simeq \ConfSink{S,\emptyset}(\Gamma_{w_1}, W-\{w_0\}).\]
  Moving elements from $S-T$ to $T$ and using the above identifications, we will show by induction
  on $|S-T|$ and the number of sinks $|W|$ that the first homology of all spaces
  $\ConfSink{S,T}(\Gamma,W)$ is generated by basic classes.

  \vspace{1em}

  In the base case, we have $T=\emptyset$ and $W=\emptyset$, so the space we are investigating is
  the ordinary configuration space $\Conf_S(G)$, which is generated by basic classes by
  \autoref{prop:1cycles-general-graph-induction} (this is not a circular argument, the proposition
  is only stated and proven later since it is the main step to compute the first homology of
  configuration spaces of arbitrary finite graphs).
  For the induction step, choose an arbitrary $s\in S-T$ and take the open covering $\{V_1, V_2\}$
  of $\ConfSink{S,T}(\Gamma_{w_0}, W)$ given by the subsets
  \[
    V_1 \defeq \pi_s^{-1}(G-\{w_0\}) \quad \text{ and } \quad
    V_2 \defeq \pi_s^{-1}(\{ x\in G\,|\, d_G(x,w_0) < 1\}).
  \]

  \begin{figure}[htpb]
    \centering
    \begin{tikzpicture}[line width = .7pt]
      \draw[-] (-1,0) -- (0.8,0);

      \draw[-] (-1,0) -- (-2.5, 1);

      \draw[-] (-1,0) -- (-2.8, 0);

      \draw[-] (-1,0) -- (-2.5,-1);

      \draw[-] plot [smooth] coordinates {(-1, 0) (-1.3, 1) (-1, 1.3) (-0.7, 1) (-1, 0)};

      \filldraw (0.8,0) circle (.1cm);
      \node at (0.8, -.4) {$w_0$};
      \filldraw (-2.5,-1) circle (.1cm);

      \draw[gray] plot [smooth cycle] coordinates {(-4.5, 0) (-2, -2) (-1, -2) (0.6, 0) (-1, 2) (-2,
      2)};
      \node at (-3.7,0) {\color{gray}\small $U_1$};

      \draw[gray] plot [smooth cycle] coordinates {(1.5, 0) (0.5, -1) (-0.7, 0) (0.5, 1)};
      \node at (1.9,0) {\color{gray}\small $U_2$};
    \end{tikzpicture}
    \caption{The open cover $\{V_1,V_2\}$ of the configuration space is defined by restricting
  particle $s$ to one of these two open sets $U_1$ and $U_2$, respectively.}
    \label{fig:open-cover-one-essential-vertex-sink}
  \end{figure}
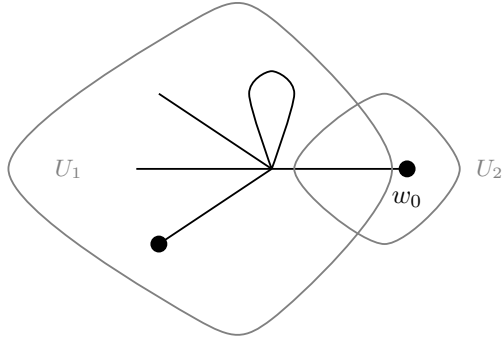

  The interesting part of the Mayer-Vietoris long exact sequence is the following:
  \begin{align*}
    H_1(V_1)\oplus H_1(V_2) &\to H_1(\ConfSink{S,T}(\Gamma_{w_0},W)) \to H_0(V_1\cap V_2) \\
    &\to H_0(V_1)\oplus H_0(V_2).
  \end{align*}
  We have $V_1\simeq \ConfSink{S, T\sqcup\{s\}}(\Gamma_{w_0},W)$, and $V_2$ is homotopy equivalent
  to a disjoint union of the space $\ConfSink{S-\{s\}, T}(\Gamma_{w_0},W)$ and several copies
  of $\ConfSink{S'}(G, W-\{w_0\})$ for different finite sets $S'\subset S$.
  Those latter components of $V_2$ arise if particles of $T$ sit between $s$ and $w_0$, preventing
  $s$ to move to the sink.
  The set $S'$ is then given by the set of all particles on the other side of $s$.
  The first component is identified by moving $s$ to the sink and forgetting it.

  The first homology of both of these spaces is by induction generated by basic classes.
  Therefore, it remains to show that the classes coming from the kernel $H_0(V_1\cap V_2)\to
  H_0(V_1)\oplus H_0(V_2)$ are generated by basic classes.

  In $V_1\cap V_2$ the particle $s$ is trapped on the edge $e$ between $w_0$ and the central vertex.
  We can represent each connected component by a configuration where all particles sit on $e$.
  The remaining particles are then distributed to both sides of $s$.
  Restricted to the connected components where there is a particle of $T$ on the $w_0$-side of $s$,
  the map
  \[
    V_1\cap V_2\hookrightarrow V_2
  \]
  is a homeomorphism onto the corresponding connected components of $V_2$ because those particles
  in $T$ prevent $s$ from moving to the sink $w_0$.
  The image of that restricted inclusion is disjoint from the image of the remaining components, so
  to find elements in the kernel of
  \[
    H_0(V_1\cap V_2)\to H_0(V_1)\oplus H_0(V_2)
  \]
  we can restrict ourselves to the union $X$ of components where no element of $T$ is on the
  $w_0$-side of $s$.

  The inclusions $X\to V_1$ and $X\to V_2$ map all these connected components to the same component
  of $V_1$ and $V_2$, respectively, because we can use either the sink or the essential vertex to
  reorder the particles.
  Therefore, the kernel of the map to $H_0(V_1)\oplus H_0(V_2)$ is generated by differences of
  distinct ways of putting particles in $S-T$ to the two sides of $s$, and the lifting
  process turns these differences into $\HH$-classes involving $w_0$ and the central vertex, proving
  the claim.
\end{proof}

\subsection{A basis for configurations in graphs with one essential vertex}
\label{sec:basis-conf-stars}
The key to proving the induction step is choosing for each leaf $e$ a particular system of
bases for all first homology groups $H_1(\ConfSink{\bullet}(G,W))$ with the following property:
if a representative of a basis element has fixed particles on the leaf $e$ then changing the order
of these particles should give another basis element, and all these basis elements should be
distinct.
Furthermore, adding and forgetting fixed particles of representatives of basis elements should again
give elements in the chosen system of bases.
For the description of such a system of bases, fix the graph $G$, the set of sinks $W$ and the leaf
$e$.

For all finite sets $S$ we will choose a system of spanning trees $T_S$ in the combinatorial model of
$\ConfSink{S}(G,W)$.
As constructed in \autoref{prop:combinatorial-model-sinks}, this model is a graph.
For each edge $\xi$ in the combinatorial model, the system $T_\bullet$ will have the following
properties:
\begin{itemize}
  \item The edge $\xi$ determines a set $F_\xi$ of fixed particles on the leaf $e$.
    The symmetric group $\Sigma_{F_\xi}\le \Sigma_n$ acts on the combinatorial model by
    precomposition, and we want that the orbit $\Sigma_{F_\xi}\cdot \xi$ is completely contained in
    either $T_S$ or $G-T_S$.
  \item Given $s\not\in S$ we have a map $\ConfSink{S}(G,W)\to\ConfSink{S\sqcup\{s\}}(G,W)$ by
    adding the particle $s$ to the end of the leaf $e$. Then $\xi$ should be in $T_S$ if and only if
    the image of $\xi$ under that map is contained in $T_{S\sqcup \{s\}}$.
\end{itemize}

We now inductively choose the system of spanning trees $T_S$.
For $S=\emptyset$, we define $T_\emptyset = \emptyset$.
Given a non-empty set $S$, complete the forest
\[
  \bigsqcup_{s\in S} \iota_{e,s}\left( T_{S-\{s\}} \right)
\]
to a spanning tree $T_S$ in an arbitrary way.
If $S'\subset S$ then $T_{S'}$ appears as subtrees of $T_S$ by adding the particles $S-S'$ to the
leaf $e$ in all different orders.
While completing this forest we only add edges that have no fixed particles on $e$, otherwise, one
of the trees $T_{S-\{s\}}$ was not maximal in $\ConfSink{S-\{s\}}(G,W)$.
This yields a spanning tree $T_S$ of $\ConfSink{S}(G,W)$, inductively describing spanning trees for
all finite sets $S$ with the properties listed above.

\vspace{1em}

This defines a system of bases $\mathcal{B}_\bullet$ of $H_1(\ConfSink{\bullet}(G,W))$ with the
following properties:
\begin{itemize}
  \item for $\sigma\in\mathcal{B}_S$ the class $\sigma^\eta$ given by adding a set of particles $T$
    in some order $\eta$ to the end of the leaf $e$ is an element of $\mathcal{B}_{S\sqcup T}$,
  \item for $\sigma\in\mathcal{B}_S$ the classes $\sigma^\eta$ and $\sigma^{\eta'}$ for two
    orderings $\eta\neq\eta'$ of $T$ are distinct,
  \item every $\sigma\in\mathcal{B}_S$ has precisely one \emph{minimal representative}
    $\sigma_{\min}\in\mathcal{B}_{S'}$ for $S'\subset S$ such that $(\sigma_{\min})^\eta=\sigma$ for
    some ordering $\eta$ of $S-S'$ (meaning that the set $S'$ is minimal with respect to this
    property) and
  \item we always have $\left( \sigma^\eta \right)_{\min} = \sigma_{\min}$.
\end{itemize}
Given $\sigma\in\mathcal{B}_S$ and the corresponding minimal cycle $C$, define $S'$ to be the
set of fixed particles of $C$ which are on $e$.
Then $\pi_{S-S'}(\sigma)$ defines the minimal representative $\sigma_{\min}\in\mathcal{B}_{S-S'}$.
With this definition it is straightforward to check the four properties described above.

\subsection{The spectral sequence for the induction step}
Let $(G,W)$ be a tree with loops with any subset of the vertices \emph{of valence one} turned into
sinks, and $v$ an essential vertex which is connected to precisely one other essential vertex $w$
via an edge $e$.
Define the following two open subspaces of $G$:
\[
  L \defeq \left\{ x\in G \,|\, \text{$d_G(x,v) < 1$} \right\}
\]
and
\[
  K \defeq \left\{ x\in G \,|\, d_G(x, G-L) < 1 \right\},
\]
where $d_G$ is the path metric giving every internal edge of $G$ length 1 and every leaf length
$1/2$.
In other words, $K$ is the connected component of $G-\{v\}$ containing $w$, see
\autoref{fig:subgraphs-adding-a-star-graph}.

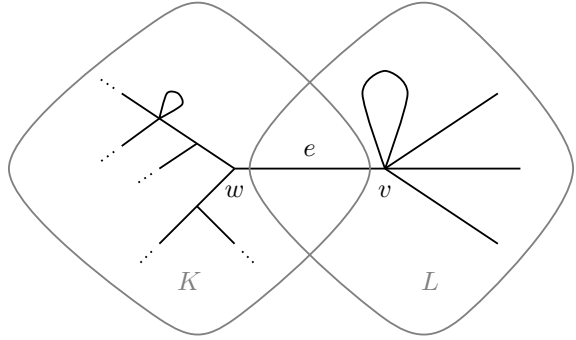
\begin{figure}[htpb]
  \centering
  \begin{tikzpicture}[line width = .7pt]
    \draw[-] (-1,0) -- (1,0);

    \draw[-] (-1,0) -- (-2.5, 1);
    \draw[dotted] (-2.6,.666667*1.6) -- (-2.8, .666667*1.8);

    \draw[-] (-2,.66667) -- (-2.5, .66667 - .5*0.73334);
    \draw[dotted] (-2.6, .66667 - .6*0.73334) -- (-2.8,.66667 - .8*0.73334);

    \draw[-] plot [smooth] coordinates {(-2, .66667) (-1.9, 1) (-1.75, 1) (-1.7, .85) (-2, .66667)};

    \draw[-] (-1.5,.33333) -- (-2, .33333 - .5*.666667);
    \draw[dotted] (-2.1,.33333 - .6*.6666667) -- (-2.3,.33333 - .8*.6666667);

    \draw[-] (-1,0) -- (-2,-1);
    \draw[dotted] (-2.1,-1.1) -- (-2.3,-1.3);

    \draw[-] (-1.5,-.5) -- (-1,-1);
    \draw[dotted] (-.9,-1.1) -- (-0.7,-1.3);

    \node at (-1.6, -1.5) {\color{gray}$K$};

    \node at (0, 0.25) {$e$};

    \draw[-] (1,0) -- (2.5, 1);

    \draw[-] (1,0) -- (2.8, 0);

    \draw[-] (1,0) -- (2.5,-1);

    \draw[-] plot [smooth] coordinates {(1, 0) (1.3, 1) (1, 1.3) (0.7, 1) (1, 0)};

    \node at (-1, -.3) {$w$};
    \node at (1, -.3) {$v$};

    \node at (1.6, -1.5) {\color{gray}$L$};

    \draw[gray] plot [smooth cycle] coordinates {(-4, 0) (-2, -2) (-1, -2) (0.8, 0) (-1, 2) (-2,
    2)};

    \draw[gray] plot [smooth cycle] coordinates {(3.5, 0) (2, -2) (1, -2) (-0.8, 0) (1, 2) (2, 2)};
  \end{tikzpicture}
  \caption{The two subgraphs $K,L$ of $G$.}
  \label{fig:subgraphs-adding-a-star-graph}
\end{figure}

The intersection $L\cap K$ is the interior of the edge $e$.
The graph $K$ has strictly fewer essential vertices than $G$, so by induction we can assume that
its configuration spaces (with sinks) of any number of particles are torsion-free and generated by
products of basic classes.

As described in \autoref{sec:mv-spectral-sequence}, construct the open cover $\mathcal{U}(\{ K,
L \})$ of $\ConfSink{n}(G,W)$ and look at the corresponding Mayer-Vietoris spectral sequence
$E^*_{\bullet,\bullet}$.
The open cover has one open set for each map $\phi\colon\mathbf{n}\to \{K, L\}$, restricting
particle $i$ to the open set $\phi(i)$.

We have
\begin{align*}
  U_{\phi_0\cdots\phi_p} &= U_{\phi_0}\cap\cdots\cap U_{\phi_p}\\
    &= \bigcap_{i\in\mathbf{n}} \bigcap_{0\le j\le p} \pi_i^{-1}\left( V_{\phi_j(i)} \right)\\
    &\simeq \coprod_{j\in J} \ConfSink{S^j_L}(L, W_L)\times \ConfSink{T^j_K}(K,W_K),
\end{align*}
where $J$ is a finite index set, $S^j_L\sqcup T^j_K\subset\mathbf{n}$ and $W_L$ and $W_K$ are the
sinks of $L$ and $K$, respectively.
To see this, notice that each connected component of such an intersection has three types of
particles:
\begin{itemize}
  \item particles which can move everywhere in $L$,
  \item particles which can move everywhere in $K$,
  \item particles which are restricted to the intersection $L\cap K$.
\end{itemize}
A particle $x$ of the last type either has $\{\phi_0(x), \ldots, \phi_p(x)\}=\{K,L\}$ or is trapped
by another particle.
Since each connected component of the configuration space of particles in the interval $L\cap K$
is contractible, we get an identification as described above simply by forgetting the particles
restricted to the intersection.
The order of the particles on this intersection will be important for the face maps given by going
from $(p+1)$-fold intersections to $p$-fold intersections by forgetting one of the open sets.

\vspace{1em}

The $E^1$-page consists at position $(p,q)$ of the $q$-th homology of all $(p+1)$-fold intersections
of the open sets $U_\phi$.
By the identification above and the Künneth theorem, each $E^1_{p,q}$ is given as
\[
  E^1_{p,q}\cong \bigoplus_{j\in J'}\bigoplus_{q_L+q_K=q} H_{q_L}(\ConfSink{S^j_L}(L, W_L)) \otimes
  H_{q_K}(\ConfSink{S^j_K}(K, W_K)),
\]
where $J'$ is some finite indexing set.
Here we used that we know that the configuration spaces of $L$ have free homology.
Recall that attached to each of those summands there is an ordering of the particles
$\mathbf{n}-S^j_L-S^j_K$, which are sitting on the interior of $e$.
The face maps forgetting one of the open sets from a $(p+1)$-fold intersection yielding a $p$-fold
intersection only affect the particles restricted to the intersection $L\cap K$:
for some (but possibly none) of them the restriction is removed, allowing them to move in all of
either $L$ or $K$.
Under the identification above, these particles are added to the sets $S^j_L$ or $S^j_K$ and put
to the edge $e$ of $L$ or $K$, respectively, in the order determined by their order on $L\cap
K$.

\vspace{1em}

Since the configuration space of $L$ is 1-dimensional by \autoref{prop:combinatorial-model-sinks}
these summands of $E^1_{p,q}$ are only non-trivial for $q_L\in\{0,1\}$.
The horizontal boundary map $d_1$ preserves $q_L$, so the $E^1$-page splits into two parts
$(^0\!E^1, ^0\!d_1)$ and $(^1\!E^1, ^1\!d_1)$ consisting of all direct summands with $q_L = 0$
and $q_L=1$, respectively.
The key point is now that $^1\!E^2$ is concentrated in the zeroth column, we understand
$^0\!E^\infty$, and the two spectral sequences don't interact.

\subsection{The homology of $^1\!E^1$}
As described in \autoref{sec:basis-conf-stars}, choose a system of bases $\mathcal{B}_\bullet$ for
$H_1(\ConfSink{\bullet}(L, W_L))$ for the edge of $L$ corresponding to $e$.
This determines a direct sum decomposition of the direct summands of every module $^1\!E^1_{p,q}$ as
follows:
\begin{align*}
  H_1(\ConfSink{S^j_L}(L, W_L)) &\otimes H_{q-1}(\ConfSink{S^j_K}(K, W_K))
  \\&\cong \bigoplus_{\sigma\in\mathcal{B}_{S^j_L}} \Int_\sigma\otimes
  H_{q-1}(\ConfSink{S^j_K}(K, W_K)).
\end{align*}
Here, $\Int_\sigma$ is the free abelian group on the single generator $\sigma$.

By the description of the face maps above and the properties of the system of bases, the boundary
map $^1\!d_1$ does not change the \emph{minimal} representative of the first tensor factor.
Grouping these summands by their corresponding minimal representative $\sigma_0$ yields a
decomposition of each row $^1\!E^1_{\bullet,q}$ into summands denoted by $(E^1[\sigma_0],
d^{\sigma_0}_1)$, which is a decomposition \emph{as chain complexes}.
We now compute the homology of one of these chain complexes $E_{\bullet,q}^1[\sigma_0]$ for fixed
$\sigma_0$ and $q\ge0$.

\vspace{1em}

Let a minimal $\sigma_0\in\mathcal{B}_S$ for some $S\subset\LST{n}$ be given (i.e.\
$(\sigma_0)_{\min}=\sigma_0$), then every $\sigma\in\mathcal{B}_{S'}$ appearing in one of the second
tensor factors of the modules in the chain complex $E^1_{\bullet,q}[\sigma_0]$ is given by adding
fixed particles $S'-S$ to $\sigma_0$, putting them in some ordering to the end of $e$ (away from
$v$).
Since there are no relations between the different orderings of the particles $S'-S$, we can forget
the particles $S$ and replace $L$ by an interval:

Let $^KE^*_{\bullet,\bullet}$ be the Mayer-Vietoris spectral sequence for $\ConfSink{\LST{n}-S}(K,
W_K)$ corresponding to the cover $\{K,L\}$ pulled back by the inclusion $K\hookrightarrow G$.
The chain complex $E^1_{\bullet,q}[\sigma_0]$ is isomorphic to the chain complex
$^KE^1_{\bullet,q}$ by forgetting the particles $S$ involved in $\sigma_0$ and looking at
cycles of the remaining particles.

The open cover of $K$ is very special: one of the open sets is the whole space itself.
We will now show that because of that, the $E^2$-page is concentrated in the zeroth column.
The open cover of $\ConfSink{\LST{n}-S}(K,W_K)$ is indexed by maps $\psi\colon\LST{n}-S\to \{K,
L\cap K \}$.
For the map $\psi_\mathrm{all}$ sending everything to $K$, we have
$U_{\psi_\mathrm{all}}=\ConfSink{\LST{n}-S}(K,W_K)$.
Hence, for each tuple $(\psi_0,\ldots,\psi_p)$ with $\psi_i\neq\psi_\mathrm{all}$ for all $i$ the
inclusion
\[
  U_{\psi_0}\cap\cdots\cap U_{\psi_p}\cap U_{\psi_\mathrm{all}} \to U_{\psi_0}\cap\cdots\cap
  U_{\psi_p}
\]
and therefore the face maps
\[
  H_q(U_{\psi_0}\cap\cdots\cap U_{\psi_p}\cap U_{\psi_\mathrm{all}}) \to
  H_q(U_{\psi_0}\cap\cdots\cap U_{\psi_p})
\]
are the identity.
Notice that precisely one of the $p+2$ face maps with that source lands in an intersection without
$U_{\psi_\mathrm{all}}$.
By adding $^K\!d_1$ boundaries we can thus assume that every homology class of the chain complex
$(^KE^1_{p,q}, {^K}d_1)$ has a representative which is trivial in all direct summands
$H_q(U_{\psi_0\cdots\psi_p})$ where none of the $\psi_i$ is $\psi_\mathrm{all}$.

The composition of maps
\[
  \bigoplus_{\substack{\psi_0<\cdots<\psi_p\\ \exists i: \psi_i=\psi_\mathrm{all}}}
  \!\!\!\!\!H_q(U_{\psi_0\cdots\psi_p}) \xrightarrow{^K\!d_1}
  \bigoplus_{\psi_0<\cdots<\psi_{p-1}} \!\!\!\!\!\!\!H_q(U_{\psi_0\cdots\psi_{p-1}})
  \twoheadrightarrow
  \bigoplus_{\substack{\psi_0<\cdots<\psi_{p-1}\\ \not\exists i: \psi_i=\psi_\mathrm{all}}}
  \!\!\!\!\!\!\!H_q(U_{\psi_0\cdots\psi_{p-1}}),
\]
where the second map collapses all direct summands with one of the $\psi_i$ equal to
$\psi_\mathrm{all}$, is injective by the observation above (actually the images of the direct
summands intersect trivially, and restricted to one such summand the map onto its image is given by
either the identity or multiplication by $-1$).
In particular, the map $^K\!d_1$ restricted to the intersections including $U_{\psi_\mathrm{all}}$ is
injective (unless we are in the zeroth degree), and the homology is trivial.

Therefore, the homology of $E^1_{\bullet,q}[\sigma_0]$ is zero in degrees $i\neq 0$ and given by
\[
  \Int_{\sigma_0}\otimes H_{q-1}\left(\ConfSink{\LST{n}-S}(K,W_K) \right)
\]
for $i=0$, which by induction is free and generated by products of basic classes.

In conclusion, the homology of $^1\!E^1$ is free, concentrated in the zeroth column and generated by
products of basic classes.
Denote this bigraded module by $E^\infty[K]$.

\subsection{The homology of $^0\!E^1$ and the $E^\infty$-page}

The other part, $^0\!E^1$, is actually the first page of the Mayer-Vietoris spectral sequence
$E_{\bullet,\bullet}^*[G/L]$ of $G$ with $L-e$ collapsed to a sink with respect to the image
of the open cover $\mathcal{U}(\{K,L\})$.
By induction, this spectral sequence $E_{\bullet,\bullet}^*[G/L]$ converges to a free infinity page,
and the corresponding homology is generated by products of basic classes.

The $E^2$-page of our original spectral sequence is hence given by the direct sum of the two
bigraded modules $E^2[G/L]$ and $E^\infty[K]$, which differs from $E^2[G/L]$ only in the zeroth
column.
We will now show that for each $2\le\ell\le\infty$ the $E^\ell$-page is the direct sum of
$E^\ell[G/L]$ and $E^\infty[K]$.

\vspace{1em}

For $p>0$ and $q\ge 0$ look at the map $d_2$ starting in $E^2_{p,q}$.
This map is constructed by representing each class in $E^2_{p,q}$ on the chain level (i.e.\ on the
$E^0$-page), mapping it via the horizontal boundary map to $E^0_{p-1,q}$, lifting it to
$E^0_{p-1,q+1}$ and applying the horizontal map again, landing in $E^0_{p-2,q+1}$.
The element of $E^2_{p-2,q+1}$ represented by this cycle is the image of the class we started with
under $d_2$.
The lifting of the particles in $L$ always connects pairs of distinct orderings of particles on
$e$ via a path through the central vertex of $L$.
The end result does not depend on the choice of such a lift, so we always take the following one:
choose (once and for all) two leaves $e_1, e_2$ of $L$ that are different from $e$, then
connecting two orderings $\nu\neq\nu'$ of a $S=\{s_1,\ldots,s_m\}$ is given by starting with the
configuration $\nu$ on $e$, sliding all particles between $s_1$ and the central vertex to $e_2$,
moving $s_1$ to $e_1$, moving the other particles back to $e$ and repeating this for all particles
$s_2,\ldots, s_m$.
Repeating the same for $\nu'$ we get two paths which glued together give a path $\gamma[\nu,\nu']$
between the two configurations.

By construction it is clear that $\gamma[\nu,\nu'] + \gamma[\nu',\nu''] = \gamma[\nu,\nu'']$, so the
only closed loop arising in such a way is the trivial path.
The construction of the image of a class under $d_2$ as described above produces segments
$\gamma[\nu,\nu']$ adding up to a cycle, which hence must be trivial.
This shows that $d_2$ maps to zero in $E^\infty[K]$ and hence that $E^3 \cong E^3[G/L] \oplus
{E^\infty[K]}$.
By the same reasoning, this is true for all pages, proving that
\[
  E^\infty \cong E^\infty[G/L] \oplus {E^\infty[K]}.
\]

In conclusion, the $E^\infty$-page is torsion-free and the corresponding homology is generated by
products of basic classes.

\begin{proof}[{Proof of \autoref{thm:trees}}]
  For graphs with precisely one vertex of valence at least three and any subset of the vertices of
  valence 1 turned into sinks the theorems follow from \autoref{prop:tree-base-case}.
  By induction on the number of essential vertices, we then use the calculation of the spectral
  sequence above to prove this for any graph as in the statement of the two theorems with any subset
  of the vertices of valence 1 turned into sinks.
  In particular, this proves the statement for the case where none of the vertices are sinks.
\end{proof}

\section{Configurations of particles in general finite graphs}\label{sec:general-graph}

In this section we prove that the \emph{first} homology of configuration spaces of graphs with rank
at least one is generated by basic classes.
In contrast to the case of trees with loops, we prove that in general
the higher homology groups are \emph{not} generated by products of 1-classes.

\subsection{The first homology of configurations in general
graphs}\label{sec:first-homology-general-graph}
For a graph $G$, we choose distinct edges $e_1, \ldots, e_\ell$ such that cutting those edges in the
middle yields a tree.
Fix identifications of $[0,1]$ with each of the $e_i$ and denote for $x\in[0,1]$ by $x_{e_i}$ the
corresponding point on the edge $e_i$.
Then, define the tree $K$ as
\[ K = G - \bigcup_{1\le i\le \ell} [1/3, 2/3]_{e_i}, \]
where $[1/3,2/3]_{e_i} = \{ x_{e_i} \,|\, x \in [1/3, 2/3] \}$.
The idea is now to start with the configuration space of $K$ embedded into the configuration space
of $G$ and to release the particles into the bigger graph $G$ one at a time.

\vspace{1em}

For $\Gamma = (G,K)$ recall the definition of $\ConfSink{S,T}(\Gamma,W)$
(\autoref{def:configurations-of-tuples}).
We will prove that $H_1(\ConfSink{S,T}(\Gamma,W))$ is always generated by basic classes.
The second part of \autoref{thm:non-product-general-graph} will be proven in the next section.
We will again proceed by constructing an open cover and investigating the Mayer-Vietoris spectral
sequence.

\vspace{1em}

Let $\ConfSink{S,T}(\Gamma, W)$ with $S-T$ non-empty be given, then choose an arbitrary element
$s\in S-T$ and construct the following open cover:
for each $i$, define two open subsets $U_{+e_i}$ and $U_{-e_i}$ of $\ConfSink{S,T}(\Gamma,W)$ by

\begin{align*}
  U_{+e_i} &= \left\{ f\colon S\to G \,|\, \text{$f(s)\not\in [1/3, 2/3]_{e_j}$ for
  $j\neq i$ and $f(s)\neq 2/3_{e_i}$} \right\}\\
  U_{-e_i} &= \left\{ f\colon S\to G \,|\, \text{$f(s)\not\in [1/3, 2/3]_{e_j}$ for
  $j\neq i$ and $f(s)\neq 1/3_{e_i}$} \right\}.
\end{align*}

\begin{figure}[htpb]
  \centering
  \begin{tikzpicture}[line width = .7pt]
    \draw[-] plot [smooth] coordinates {(2, 0) (2.6, .2) (2.8, 0) (2.6, -.2) (2, 0)};

    \draw[pattern=north east lines,pattern color=lightgray] plot [smooth cycle] coordinates {
      (-4, 0)
      (-3, 1)
      (0.7, 1.2)
      (1, 0.2)
      (1.4, 1.1)
      (2.6, 0)
      (1.4, -1.1)
      (0, -.3)
      (-1.4, -1.1)
      (-3, -1)
    };

    \draw[-] (-2,0) -- (2,0);
    \draw[-] plot [smooth] coordinates {(-2, 0) (-1, .6) (0, .8) (1, .6) (2, 0)};
    \draw[-] plot [smooth] coordinates {(-2, 0) (-1, -.6) (0, -.8) (1, -.6) (2, 0)};
    \draw[-] (-2,0) -- (-3,0);
    \draw[-] (-2,0) -- (-2.7,0.7);
    \draw[-] (-2,0) -- (-2.7,-0.7);

    \node at (0, 1) {$e_1$};
    \node at (0, -1) {$e_2$};
    \node at (3.1, 0) {$e_3$};

  \end{tikzpicture}
  \caption{The part of $G$ where the particle $s$ is allowed in the open set $U_{+e_1}$, where $e_1$
  is oriented from left to right.}
  \label{fig:open-set-u-plus-e-one}
\end{figure}
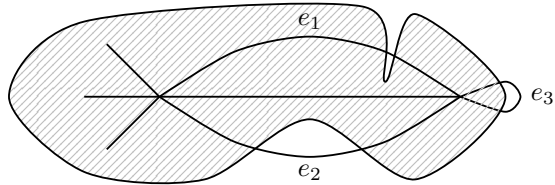

Let $T' = T\sqcup \{s\}$ and $\Gamma'=(G-[1/3,2/3]_{e_i},K)$.
\begin{prop}\label{prop:identification-intersection-general-graph}
  The intersections of those open sets can be identified as follows:
  \begin{align*}
    U_{\pm e_i} &\simeq \ConfSink{S,T'}(\Gamma, W)\\
    U_{-e_i}\cap U_{+e_i} &\simeq \ConfSink{S,T'}(\Gamma,W) \sqcup \ConfSink{S-\{s\}, T}(\Gamma',W)\\
    U_{\pm e_i}\cap U_{\pm e_j} &\simeq \ConfSink{S,T'}(\Gamma, W).
  \end{align*}
  Any intersection of at least three of those open sets is again homotopy equivalent to
  $\ConfSink{S,T'}(\Gamma,W)$.

  The inclusions induced by going from $p$-fold intersections to $(p-1)$-fold intersections are
  homotopic to the identity on the components $\ConfSink{S,T'}(\Gamma,W)$ and given by adding the
  particle $s$ to $1/2_{e_i}$ for the configurations in each component
  $\ConfSink{S-\{s\},T}(\Gamma',W)$.
  These latter components are not hit by any such inclusion.
\end{prop}
\begin{proof}
  If the intersection of any number of these open sets contains open sets $U_{\pm e_i}$ and
  $U_{\pm e_j}$ for $i\neq j$ then the particle $s$ is restricted from entering all
  $[1/3,2/3]_{e_i}$, so this intersection is actually precisely the same as
  $\ConfSink{S,T'}(\Gamma,W)$.
  Since every intersection of $\ge 3$ of those sets contains two such open sets, there are only two
  cases remaining, namely 1-fold intersections and the intersection $U_{-e_i}\cap U_{+e_i}$.

  The space $U_{+e_i}$ is almost the same as $\ConfSink{S,T'}(\Gamma,W)$, the only difference is
  that the particle $s$ is also allowed in the segment $[1/3, 2/3)_{e_i}$.
  By sliding $s$ back into the interval $[0, 1/3)_{e_i}$ whenever necessary and moving all particles
  between $0_{e_i}$ and $s$ accordingly, we see that this space is homotopy equivalent to
  $\ConfSink{S,T'}(\Gamma,W)$.
  The analogous reasoning identifies $U_{-e_i}$.

  The intersection $U_{-e_i}\cap U_{+e_i}$ has two connected components:
  the component where $s$ is in $(1/3,2/3)_{e_i}$ and the one where it is in $K$.
  The second component is again on the nose equal to $\ConfSink{S,T'}(\Gamma,W)$.
  Modify the first component by a homotopy moving $s$ to $1/2_{e_i}$ and sliding all other particles
  on $e_i$ away from $s$ into the intervals $[0, 1/3)_{e_i}$ and $(2/3, 1]_{e_i}$, then forgetting
  the particle $s$ gives an identification with $\ConfSink{S-\{s\},T}(\Gamma',W)$, proving the first
  claim.

  \vspace{1em}

  By our identification above the description of the inclusion maps given by forgetting one of the
  intersecting open sets is easily deduced.
  If one of these inclusions would hit a component $\ConfSink{S-\{s\},T}(\Gamma',W)$, then
  the particle $s$ would need to be on the interval $(1/3, 2/3)_{e_i}$, which it never is for any
  triple intersection.
\end{proof}

This allows us to describe generators for the first homology of the configuration space of any
finite graph.
We formulate this as a separate proposition in order to use it for the case where $K$ is a graph
with precisely one essential vertex since this case is needed to prove \autoref{thm:trees}.
\begin{prop}\label{prop:1cycles-general-graph-induction}
  Let $G$ be a connected finite graph, $K\subset G$ a tree defined as above and $W$ a subset of the
  vertices.
  If $H_1(\ConfSink{S}(K, W))$ is generated by basic classes for all finite sets $S$ then also
  $H_1(\ConfSink{S,T}(\Gamma,W))$ is generated by basic classes for all pairs of finite sets
  $T\subset S$, where $\Gamma=(G,K)$.
\end{prop}
\begin{proof}
  We prove this by looking at the spectral sequence constructed from the open cover described above.
  To prove the statement we only need to show that moving one element out of $T$ preserves the
  property that the homology is generated by basic classes.
  We can assume that the configuration space of $K$ is connected since the only case where this is
  not true is if $G$ is $S^1$ without sinks, and this case is true by definition.
  We will now argue by induction on the number of elements in $S-T$.
  The induction start $S=T$ is precisely that $H_1(\ConfSink{S}(K,W))$ is generated by basic
  classes, so we only need to check the induction step.

  In the induction step, we only get 1-classes at $E^\infty_{0,1}$ and $E^\infty_{1,0}$.
  The module $E^\infty_{0,1}$ is a quotient of $E^1_{0,1}$, which is generated by 1-classes of
  $U_{\pm e_i}\simeq\ConfSink{S,T'}(\Gamma,W)$, so by induction by classes of the required form.

  The chain complex $E^1_{\bullet,0}$ is given by the chain complex of the nerve of the cover (which
  is a simplex) and one additional copy of $\Int$ for each intersection $U_{+e_i}\cap U_{-e_i}$.
  Restricted to $H_0(U_{-e_i}\cap U_{+e_i})\cong \Int\oplus\Int$ the face maps
  \[
    \Int\oplus\Int \cong H_0(U_{-e_i}\cap U_{+e_i}) \to H_0(U_{\pm e_i}) \cong \Int
  \]
  are given by $(x,y)\mapsto \pm(x+y)$.
  Therefore, all elements $(x, -x)$ are in the kernel of $d_1$.
  These elements correspond to $S^1$ movements of $s$ along the edge $e_i$:
  by mapping $U_{-e_i}\cap U_{+e_i}\hookrightarrow U_{-e_i}$ the particle $s$ is allowed to leave
  $(1/3,2/3)_{e_i}$ via one of the sides, connecting it to a configuration where $s$ is on the tree
  $K$.
  The other inclusion allows $s$ to leave via the other side, connecting it to that same
  configuration with $s$ on $K$.
  Mapping this to $\ConfSink{S,T}(\Gamma,W)$ yields a cycle where $s$ moves along $K$ and
  $e_i$. We can choose a representative such that all other particles are fixed and that this
  movement follows an embedded circle in $G$.

  Subtracting such kernel elements, we can modify every cycle of $(E^1_{\bullet,0}, d_1)$ such that
  it is zero in all copies of $H_0(\ConfSink{S-\{s\},T}(\Gamma',W))$.
  Since the remaining part of the chain complex is the chain complex of a simplex, there are no
  other 1-classes, concluding the argument.
\end{proof}

\begin{proof}[{Proof of \autoref{thm:non-product-general-graph} --- first homology group}]
  By \autoref{thm:trees}, the homology group $H_1(\Conf_{S}(K))$ is generated by basic
  classes for any finite tree $K$, so the theorem follows from
  \autoref{prop:1cycles-general-graph-induction}.
\end{proof}

\subsection{Non-product generators}
In this section, we describe an example of a homology class of the configuration space of a graph
that cannot be written as a sum of product classes.

The easiest example we were able to find so far is a 2-class of $\Conf_3\left( B_3 \right)$, where
$B_3$ is the banana graph of rank three, i.e.\ two vertices $v,w$ connected via four edges, see
\autoref{fig:banana-graph-with-star-graph}.

\vspace{1em}

To construct the class, we first construct classes in $\Conf_2(\Star_4)$.
Let $S\subset\mathbf{3}$ be a set of two particles, then the first homology group of
$\Conf_S(\Star_3)$ is one-dimensional, a generator can be represented by a sum of twelve edges, each
with coefficient +1:
start with both particles on different edges, then in turns move the particles to the free edge
until the initial configuration is restored.

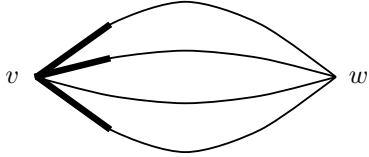
\begin{figure}[htpb]
  \centering
  \begin{tikzpicture}[line width = .7pt]
    \draw[-] plot [smooth] coordinates {(-2, 0) (-1, 0.7) (0, 1) (1, 0.7) (2, 0)};
    \draw[-] plot [smooth] coordinates {(-2, 0) (-1, 0.25) (0, 0.35) (1, 0.25) (2, 0)};
    \draw[-] plot [smooth] coordinates {(-2, 0) (-1, -0.25) (0, -0.35) (1, -0.25) (2, 0)};
    \draw[-] plot [smooth] coordinates {(-2, 0) (-1, -0.7) (0, -1) (1, -0.7) (2, 0)};
    \node at (-2.3, 0) {$v$};
    \node at (2.3, 0) {$w$};

    \draw[-, line width=2.5pt] plot [smooth] coordinates {(-1.98, 0) (-1, 0.7)};
    \draw[-, line width=2.5pt] plot [smooth] coordinates {(-2.008, 0) (-1, 0.25)};
    \draw[-, line width=2.5pt] plot [smooth] coordinates {(-1.97, 0) (-1, -0.7)};
  \end{tikzpicture}
  \caption{Including $\Star_3$ into the banana graph $B_3$ at $v$ in one of four ways.}
  \label{fig:banana-graph-with-star-graph}
\end{figure}

Now choose a bijection of $\mathbf{3}$ with the leaves of $\Star_3$ and $\mathbf{4}$ with the leaves
of $\Star_4$.
This defines four 1-cycles in $\Conf_S(\Star_4)$ by including $\Star_3$ into $\Star_4$ in all
order-preserving ways (with respect to these identifications).
Now we add those four cycles together with the following signs:
each inclusion of $\Star_3$ is determined by the edge $i\in\mathbf{4}$ that is missed.
The 1-cycle corresponding to this $i$ gets the sign $(-1)^i$.
This sum is actually equal to zero:

The 1-cells of these cycles are given by one particle moving from one edge to the central vertex and
the other particle sitting on another edge.
Each such cell appears precisely twice, once for each way of choosing a third edge from the
remaining two leaves.
If these two remaining leaves are cyclically consecutive in $\mathbf{4}$ the corresponding cycles
have different signs, otherwise, these two cells inside the 1-cycles appear with different signs, so
in both cases, they add up to zero.

Including $\Star_4$ into $B_3$ (mapping the central vertex to $v$) gives a sum of four 1-cycles
coming from embedding $\Star_3$ into $B_3$ in different ways (see
\autoref{fig:banana-graph-with-star-graph}), which evaluates to zero.

\vspace{1em}

Now let $t$ be the third particle, i.e.\ $S\sqcup \{t\} = \mathbf{3}$, then take for each of those
four 1-cycles in $\Conf_S(B_3)$ the product of the cycle with the 1-cell moving particle $t$ from the
remaining one of the four edges to the vertex $v$.

Doing this construction for all three choices of $S$ gives a sum of 144 2-cells, and the claim is
that this is, in fact, a 2-cycle in the combinatorial model of the configuration space.
We can think of this cycle as 12 cylinders of a 1-cycle in the star of $v$ multiplied with another
particle moving to the other vertex $w$, whose boundary 1-cells get identified in a certain way, see
\autoref{fig:banana-graph-with-two-cell}.

Let $t\in\mathbf{3}$, then one part of the boundary of four of those cylinders is given by the
1-cycles of the particles $\mathbf{3}-\{t\}$ with $t$ sitting on $w$.
By construction, those four 1-cycles add up to zero.

It remains to investigate the parts where the third particle is in the middle of the edge.
These 1-cells are precisely given by two particles sitting in the middle of two edges and a third
particle moving from another edge to $v$.
Each such cell appears precisely twice: once for every choice of which one of the fixed particles
moves to $w$ and which one belongs to the star movement.
By analogous reasoning, these two occurrences have opposite signs, so the total contribution is
zero.

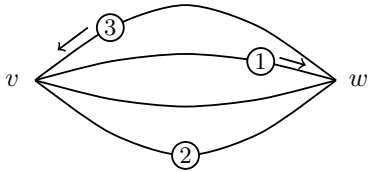
\begin{figure}[htpb]
  \centering
  \begin{tikzpicture}[line width = .7pt]
    \draw[-] plot [smooth] coordinates {(-2, 0) (-1, 0.7) (0, 1) (1, 0.7) (2, 0)};
    \draw[-] plot [smooth] coordinates {(-2, 0) (-1, 0.25) (0, 0.35) (1, 0.25) (2, 0)};
    \draw[-] plot [smooth] coordinates {(-2, 0) (-1, -0.25) (0, -0.35) (1, -0.25) (2, 0)};
    \draw[-] plot [smooth] coordinates {(-2, 0) (-1, -0.7) (0, -1) (1, -0.7) (2, 0)};
    \draw[->] (-1.3, .7) -- (-1.7, 0.4);
    \draw[->] (1.25, .3) -- (1.6, 0.2);
    \node at (-2.3, 0) {$v$};
    \node at (2.3, 0) {$w$};
    \particleNr{(1, 0.25)}{1}
    \particleNr{(-1, 0.7)}{3}
    \particleNr{(0, -1)}{2}
  \end{tikzpicture}
  \caption{Each of the twelve cylinders making up the cycle is given by twelve two cells of this
  form, where all particles are on different edges.}
  \label{fig:banana-graph-with-two-cell}
\end{figure}

\vspace{1em}

Thus, the boundary cells of the twelve cylinders add up to zero, yielding a non-trivial cycle.
By the dimension of our combinatorial model, there are no three-cells, so this does not represent
the zero class.
Notice that there are no product classes since every $S^1$ generator uses both vertices and there
are too few particles for two $\HH$-classes or star classes.
By looking at the identifications and calculating the Euler characteristic, one sees that the
resulting cycle is, in fact, a closed surface of genus 13 embedded into the combinatorial model of
the configuration space.
In fact, by pushing in 2-cells where strictly less than three edges are involved (starting with
those involving only one edge, followed by those involving precisely two edges) and afterward
pushing in the 1-dimensional intervals where particles move to an occupied edge it is
straightforward to show the following:
\begin{prop}
  $\Conf_3(B_3)$ is homotopy equivalent (equivariantly with respect to the action of the symmetric
  group $\Sigma_3$) to a closed surface of genus 13.
\end{prop}

\begin{rem}
  In between versions of this paper, Wiltshire-Gordon independently showed this homotopy equivalence
  using explicit computer calculations of the groups $H_*(\Conf_3(B_3))$, see \cite[Example 2.1, p.
  4]{WiGo17}.
\end{rem}

\vspace{1em}

We now prove the rest of \autoref{thm:non-product-general-graph}, whose first part was proven in
\autoref{sec:first-homology-general-graph}.
\begin{proof}[{Proof of \autoref{thm:non-product-general-graph} --- non-product classes}]
  A counterexample for the second homology group was described above, all that remains is to
  describe how to use this to construct counterexamples for higher homology groups.

  By adding $k$ disjoint $S^1$ graphs, connecting each of them to $v$ via a single edge and adding
  $k$ particles we can take the product of this non-product cycle with the $k$-cycle given by the
  product of the $k$ particles moving inside the $S^1$'s.
  This gives a  class in the $(k+2)$-nd homology group of the configuration space of $k+3$ particles
  in this graph, which by analogous reasoning cannot be written as a sum of product classes.
  This shows that this phenomenon appears in every homology degree except for the zeroth and first.
\end{proof}

\bibliographystyle{halpha}
\bibliography{conf-graph}
\end{document}